\documentclass[12pt,twoside]{amsart}
\usepackage{amsmath, amsthm, amscd, amsfonts, amssymb, graphicx, mathabx}

\usepackage{enumerate}
\usepackage[colorlinks=true,
linkcolor=blue,
urlcolor=cyan,
citecolor=red]{hyperref}
\usepackage{mathrsfs}
\newtheorem{theorem}{Theorem}[section]
\newtheorem{lemma}{Lemma}[section]
\newtheorem{remark}{Remark}[section]

\newtheorem{corollary}{Corollary}[section]
\newtheorem{example}{Example}[section]
\newtheorem{proposition}{Proposition}[section]
\numberwithin{equation}{section}

\begin{document}
\title{Subadditive inequalities for operators}
\author{Hamid Reza Moradi,  Zahra Heydarbeygi and Mohammad Sababheh}
 \maketitle
\begin{abstract}
In this article, we present a new subadditivity  behavior of convex and concave functions, when applied to Hilbert space operators. For example, under suitable assumptions on the spectrum of the positive operators $A$ and $B$, we prove that
\[{{2}^{1-r}}{{\left( A+B \right)}^{r}}\le {{A}^{r}}+{{B}^{r}}\quad\text{ for }r>1\text{ and }r<0,\]
and
\[{{A}^{r}}+{{B}^{r}}\le {{2}^{1-r}}{{\left( A+B \right)}^{r}}\quad\text{ for }r\in \left[ 0,1 \right].\]
These results provide considerable generalization of earlier results by Aujla and Silva. 

Further, we present several extensions of the subadditivity idea initiated by Ando and Zhan then extended by Bourin and  Uchiyama.
\end{abstract}
%------------------------------------------------------------------------------------%
\pagestyle{myheadings}
\markboth{\centerline {H.R. Moradi,  Z. Heydarbeygi \& M. Sababheh }}
{\centerline {Subadditive inequalities for operators}}
\bigskip
\bigskip
%------------------------------------------------------------------------------------%
%------------------------------------------------------------------------------------

\section{Introduction}\label{sec1}
Let $\mathcal{B}\left( \mathcal{H} \right)$ be the algebra of all (bounded linear) operators on a  complex Hilbert space $\mathcal{H}$. An operator $A$ on $\mathcal{H}$ is said to be {\it positive} (in symbol: $A\ge 0$) if $\left\langle Ax ,x  \right\rangle \ge 0$ for all $x \in \mathcal{H}$. We write $A>0$ if $A$ is positive and invertible; and we say that $A$ is strictly positive in this case. For self-adjoint operators $A$ and $B$, we write $A\ge B$ if $A-B$ is positive, i.e., $\left\langle Ax ,x  \right\rangle \ge \left\langle Bx ,x  \right\rangle $ for all $x \in \mathcal{H}$. We call it the usual order. In particular, for some scalars $m$ and $M$, we write $m\le A\le M$ if $m\left\langle x ,x  \right\rangle \le \left\langle Ax ,x  \right\rangle \le M\left\langle x ,x  \right\rangle $ for all $x \in \mathcal{H}$.  In what follows we denote the weighted arithmetic mean and the weighted geometric mean of strictly positive operators $A$ and $B$ by $A{{\nabla }_{v}}B\equiv \left( 1-v \right)A+vB$ and $A{{\sharp}_{v}}B\equiv {{A}^{\frac{1}{2}}}{{\left( {{A}^{-\frac{1}{2}}}B{{A}^{-\frac{1}{2}}} \right)}^{v}}{{A}^{\frac{1}{2}}}$, respectively. For the case $v={1}/{2}\;$ we write $A\nabla B$ and $A\sharp B$. Notice that the definition of $A{{\nabla }_{v}}B$ is still valid for positive (not necessarily strictly positive) operators.

 A real-valued continuous function $f$ on an interval $J$ is said to be {\it operator convex} (resp. {\it operator concave}) if $f\left( A{{\nabla }_{v}}B \right)\le \left( \text{resp}\text{. }\ge  \right)f\left( A \right){{\nabla }_{v}}f\left( B \right)$ for all $v\in \left[ 0,1 \right]$  and for all self-adjoint operators $A,B\in\mathcal{B}(\mathcal{H})$ whose spectra are contained in $J$. A continuous function $f$ on $J$ is called {\it operator monotone increasing} (resp. {\it decreasing}), if
\[A\le B\quad\text{ }\Rightarrow \quad\text{ }f\left( A \right)\le \left( \text{resp}\text{. }\ge  \right)f\left( B \right).\]

An important observation that relates operator convex and operator monotone functions is given by the following result \cite[Theorem 2.1, Theorem 3.1, Theorem 2.3 and Theorem 3.7]{ando1}. 
\begin{proposition}\label{oper_intro_prop}
Let $f:(0,\infty)\to [0,\infty)$ be continuous. Then 
\begin{enumerate}
\item $f$ is operator monotone decreasing if and only if $f$ is operator convex and $f(\infty)<\infty$.
\item $f$ is operator monotone increasing if and only if $f$ is operator concave.
\end{enumerate}
\end{proposition}
For a more recent reference for this proposition, we refer the reader to \cite[Theorem 2.4]{Uchiyama}. We remark that in this reference, it is shown that a function $f$ defined on $(a,\infty)$ is operator monotone if and only if $f$ is operator concave and $f(\infty)>-\infty.$ We notice that the condition $f(\infty)>-\infty$ is implicitly assumed in Proposition \ref{oper_intro_prop} since $f$ is non-negative.

Proposition \ref{oper_intro_prop} will be used frequently in the sequel; where treatment of operator monotone functions and that of operator convexity are interchangeable.\\
It is well-known that a concave function $f$ with $f(0)\geq 0$ is subadditive in the sense that 
\begin{equation}\label{conc_sub_intro}
f(a+b)\leq f(a)+f(b)
\end{equation} 
 for the non-negative numbers $a,b.$ A similar inequality is not necessarily valid for operator concave functions. That is, an operator concave function $f$ does not necessarily satisfy 
$$f(A+B)\leq f(A)+f(B)$$ for the positive operators $A,B.$ In 1999 Ando and Zhan \cite{ando} proved a subadditivity inequality for operator concave functions, which says that if $A$, $B$ are two positive  matrices, then
\begin{equation}\label{17}
{{\left\| f\left( A+B \right) \right\|}_{u}}\le {{\left\| f\left( A \right)+f\left( B \right) \right\|}_{u}}
\end{equation}
for any unitarily invariant norm ${{\left\| \cdot \right\|}_{u}}$ and any non-negative operator monotone function $f$ on $\left[ 0,\infty  \right)$. Bourin and Uchiyama \cite{bourin} showed that the condition operator concavity in \eqref{17} can be replaced by scalar concavity.

 We  refer the reader to \cite{ando,bourin} for different variants of \eqref{17}. Our main target in this paper is to discuss  subadditivity inequalities for convex and operator convex functions, without appealing to unitarily invariant norms. However, this will be at the cost of additional conditions or weaker estimates. The first result in this direction will be an operator convexity behavior for convex functions. More precisely, we prove that a convex function satisfies the operator convexity inequality
$$f\left(\frac{A+B}{2}\right)\leq \frac{f(A)+f(B)}{2},$$ under some conditions on the spectra of the positive operators $A,B.$

Moreover, we show that concave functions (not necessarily operator concave) satisfy the operator subadditivity inequality
$$k\;f(A+B)\leq f(A)+f(B),$$ for a positive scalar $k\leq 1.$ We consider this an interesting extension of \eqref{17}.

Another result of this type is due to Aujla \cite{2} which asserts that if the function $f$ is completely monotone (in the sense, ${{\left( -1 \right)}^{k}}{{f}^{\left( k \right)}}\left( x \right)\ge 0$ for all $k=0,1,\ldots $ and all $x\in \left[ 0,\infty  \right)$, where ${{f}^{\left( 0 \right)}}=0$ and ${{f}^{\left( k \right)}}$ denotes the $k$-th derivative of $f$) on $\left[ 0,\infty  \right)$, then
\begin{equation}\label{18}
2{{\left\| f\left( A+B \right) \right\|}_{u}}\le {{\left\| f\left( 2A \right)+f\left( 2B \right) \right\|}_{u}}.
\end{equation}

In Sec. \ref{sec2}, we extend the norm order in \eqref{18} to the operator order (see Corollary \ref{7}). Naturally, this generalization imposes additional conditions. Inspired by the result given in Theorem \ref{6}, we present some extensions of the inequalities due to Aujla and Silva \cite{1}. Our inequalities refine earlier results in this direction obtained in \cite{2} and \cite{kian}.
\section{Main Results}\label{sec2}
Our first main result proposes a mild condition under which operator convexity follows from scalar convexity.
\begin{theorem}\label{6}
Let $A,B\in \mathcal{B}\left( \mathcal{H} \right)$ be two positive operators such that $n\le A\le N$ and $m\le B\le M$ 
for some positive real numbers $n<N$ and $m<M$. Further, let $f:(0,\infty)\to\mathbb [0,\infty)$ be a convex function. If, for some $v\in (0,1)$,
$$\left[ n{{\nabla }_{v}}m,N{{\nabla }_{v}}M \right]\cap \left[ n,N \right],\left[ n{{\nabla }_{v}}m,N{{\nabla }_{v}}M \right]\cap \left[ m,M \right]=\varnothing,$$ 
then 
\begin{equation}\label{3}
f\left( A{{\nabla }_{v}}B \right)\le f\left( A \right){{\nabla }_{v}}f\left( B \right).
\end{equation}
In particular,
\begin{equation}\label{4}
f\left( \frac{A+B}{2} \right)\le \frac{f\left( A \right)+f\left( B \right)}{2}
\end{equation}
provided that the above empty intersection condition is fulfilled when $v=\frac{1}{2}.$\\
The reverse inequalities hold when $f$ is concave.
\end{theorem}
\begin{proof}
We use an idea from \cite[Theorem 1]{3}. For simplicity, let  $a=n\nabla_v m$ and $b=N\nabla_v M.$

Now since $[a,b]\cap [m,M]=\varnothing$ and $[a,b]\cap [n,N]=\varnothing$, we will consider the secant of $f$ on the interval $[a,b]$. So, let
$$y(t)=\frac{b-t}{b-a}f(a)+\frac{t-a}{b-a}f(b).$$
Since $f$ is convex, it satisfies
\begin{equation}\label{needed_1}
f(t)\leq y(t),\quad a\leq t\leq b.
\end{equation}
Since $n\leq A\leq N$ and $m\leq B\leq M$, it follows that $n\nabla_v m\leq A\nabla_v B\leq N\nabla_v M$, and hence 
\begin{equation}\label{needed_2}
f(A\nabla_v B)\leq y(A\nabla_v B)
\end{equation}
by applying a functional calculus argument to \eqref{needed_1}.

 Further, by the empty intersection assumption, we have  
$$f(t)\geq y(t),\quad t\in [n,N]\cup [m,M],$$ because $f$ is convex on $(0,\infty)$.

Since $n\leq A\leq N$ and $m\leq B\leq M$, functional calculus implies
$$f(A)\geq y(A)\quad{\text{ and }}\quad f(B)\geq y(B).$$
Noting that $y(A)\nabla_v y(B)=y(A\nabla_v B),$ the above inequalities together with \eqref{needed_2} imply

$$f(A)\nabla_v f(B)\geq y(A)\nabla_v y(B)= y(A\nabla_v B)\geq f(A\nabla_vB),$$ which completes the proof.
\end{proof}

Related to Theorem \ref{6}, we present the following version for $v\geq 1.$ Proceeding similarly, one can obtain similar results for $v\leq 0.$
\begin{theorem}\label{them_neg}
Let $A,B\in \mathcal{B}\left( \mathcal{H} \right)$ be two positive operators such that $n\le A\le N$ and $m\le B\le M$ 
for some positive real numbers $n<N$ and $m<M$. Further, let $f:(0,\infty)\to\mathbb [0,\infty)$ be a convex function. If, for some $v\geq 1$,
$$\left[ N{{\nabla }_{v}}m,n{{\nabla }_{v}}M \right]\cap \left[ m,M \right]=\varnothing\quad{\text{ and }}\quad[ n,N]\subset \left[ N{{\nabla }_{v}}m,n{{\nabla }_{v}}M \right],$$ 
then 
$$
f\left( A{{\nabla }_{v}}B \right)\le f\left( A \right){{\nabla }_{v}}f\left( B \right).
$$
The reverse inequalities hold when $f$ is concave.
\end{theorem}
\begin{proof}
Notice first that when $n\leq A\leq N, m\leq B\leq M$ and $v\geq 1$, then
$$c:=N\nabla_v m\leq  A\nabla_v B\leq n\nabla_v M:=d.$$ Let $$y(t)=\frac{d-t}{d-c}f(c)+\frac{t-c}{d-c}f(d).$$ Then noting the assumptions and applying a functional calculus argument, similar to Theorem \ref{6}, imply the desired inequality.
\end{proof}

\begin{remark}
In Theorem \ref{6}, we assumed that $f$ is convex on $(0,\infty).$ In fact, it is enough to assume convexity on an interval $J$ containing the three intervals $[n,N], [m,M]$ and $[n\nabla_v m, N\nabla_v M].$

Also, in Theorem \ref{6} we assumed that $A$ and $B$ are positive operators. It is clear that self adjointeness is enough.
\end{remark}
As an immediate consequence of Theorem \ref{6}, we have the following subadditivity result for convex functions.
\begin{corollary}\label{7}
Under the same assumptions of theorem \ref{6}, with $v=\frac{1}{2}$, the following subadditivity inequality holds for  the convex function $f$.
	\[2f\left( A+B \right)\le f\left( 2A \right)+f\left( 2B \right)\Leftrightarrow f\left(\frac{A+B}{2}\right)\leq \frac{f(A)+f(B)}{2}.\]
 If in addition $f\left( 2t \right)\le 2f\left( t \right)$, then
\[f\left( A+B \right)\le f\left( A \right)+f\left( B \right).\]
\end{corollary}
\begin{proof}
	We have
\[2f\left( A+B \right)=2f\left( \frac{2A+2B}{2} \right)\le f\left( 2A \right)+f\left( 2B \right).\]
This completes the proof.
\end{proof}
Notice that  the assumption $f\left( 2t \right)\le 2f\left( t \right)$ can be dropped for the non-negative decreasing functions. This follows from the fact that each non-negative decreasing function $f$ satisfies $f\left( 2t \right)\le 2f\left( t \right)$.
\begin{remark}\label{8}
Although our assumptions in Corollary \ref{7} are different from \cite[Corollary 2.5]{1}, we present an operator inequality which is much stronger than the eigenvalue inequality in \cite{1}.
\end{remark}
The following corollary presents a stronger version of \cite[Corollary 2.5]{2}. In \cite{2}, unitarily invariant versions for complex matrices have been obtained. In the following result, we present an operator version under some conditions on the spectra. In this result, we use the notation $\tau(X)$ to denote the smallest closed interval containing the spectrum of the operator $X\in\mathcal{B}(\mathcal{H})$.
\begin{corollary}\label{5}
Let $A,B\in \mathcal{B}\left( \mathcal{H} \right)$ be two strictly positive operators. Then

\[{{2}^{1-r}}{{\left( A+B \right)}^{r}}\le {{A}^{r}}+{{B}^{r}}\quad\text{ for }r>1\text{ and }r<0,\]
and
\[{{A}^{r}}+{{B}^{r}}\le {{2}^{1-r}}{{\left( A+B \right)}^{r}}\quad\text{ for }r\in \left[ 0,1 \right]\]
whenever $\tau \left( A \right)\cap \tau \left( \frac{A+B}{2}\; \right),\tau \left( B \right)\cap \tau \left( \frac{A+B}{2}\; \right)=\varnothing $.
\end{corollary}

\begin{remark}
The inequalities in Corollary \ref{5} are also stronger than the one given in \cite[Remark 2.11]{1} because of the same reasoning mentioned in Remark \ref{8}.
\end{remark}

The following example shows how Corollary \ref{5} works.
\begin{example}
Taking $A=\left( \begin{matrix}
3 & 1  \\
1 & 5  \\
\end{matrix} \right)$, $B=\left( \begin{matrix}
10 & -1  \\
-1 & 9  \\
\end{matrix} \right)$. 
It is not hard to check that $A$ and $B$ satisfy in the conditions of Corollary \ref{5}.
\begin{itemize}
	\item[(i)] For $r=6$,
	\[{{A}^{r}}+{{B}^{r}}-{{2}^{1-r}}{{\left( A+B \right)}^{r}}\approx \left( \begin{matrix}
	985931.21 & -476992  \\
	-476992 & 433279  \\
	\end{matrix} \right)\gneqq0.\]
	
	\item[(ii)] For $r=-2$,
	\[{{A}^{r}}+{{B}^{r}}-{{2}^{1-r}}{{\left( A+B \right)}^{r}}\approx \left( \begin{matrix}
	0.0956 & -0.0384  \\
	-0.0384 & 0.0229  \\
	\end{matrix} \right)\gneqq0.\]
		
	\item[(iii)] For $r={1}/{3}\;$,
	\[{{2}^{1-r}}{{\left( A+B \right)}^{r}}-{{A}^{r}}-{{B}^{r}}\approx\left( \begin{matrix}
	0.1519 & -0.061  \\
	-0.061 & 0.0486  \\
	\end{matrix} \right)\gneqq0.\]
\end{itemize}
\end{example}

Although Corollary \ref{5} does not present equivalent conditions; and only necessary conditions are proposed, we present the following example which presents an example where the conditions and the conclusions are not satisfied. 
\begin{example}
Let 
$$A=\left(\begin{array}{cc}1&1\\1&1\end{array}\right),\quad B=\left(\begin{array}{cc}3&1\\1&1\end{array}\right).$$
This example was given in \cite[Example V.1.4., P. 114]{B2} to show that the function $f(t)=t^3$ is not operator convex. That is, it is shown there that the inequality
\[{{2}^{1-r}}{{\left( A+B \right)}^{r}}\le {{A}^{r}}+{{B}^{r}}\] is not true for $r=3$ and the above matrices.
Here we explain why this inequality of Corollary \ref{5} does not apply. The reason lies in the computations of the spectra:
$$\tau(A)=[0,2],\;\tau(B)=[2-\sqrt{2},2+\sqrt{2}]\;{\text{and}}\;\tau\left(\frac{A+B}{2}\right)=\left[\frac{3-\sqrt{5}}{2},\frac{3+\sqrt{5}}{2}\right].$$
Clearly, $$\tau(A)\cap \tau\left(\frac{A+B}{2}\right)\neq\varnothing \;{\text{and}}\;\tau(B)\cap \tau\left(\frac{A+B}{2}\right)\neq\varnothing.$$
That is, the conditions of Corollary \ref{5} are not satisfied . It is also readily seen that the inequality ${{2}^{1-r}}{{\left( A+B \right)}^{r}}\le {{A}^{r}}+{{B}^{r}}$ is not valid, for $r=3.$
\end{example}

Our next result is a Hermite-Hadamard type inequality for operators satisfying certain conditions on their spectra.

\begin{corollary}
Let $A,B\in \mathcal{B}\left( \mathcal{H} \right)$ be two positive operators such that $n\le A\le N$ and $m\le B\le M$ 
for some positive real numbers $n<N$ and $m<M$. Further, let $f:(0,\infty)\to\mathbb [0,\infty)$ be a convex function. If, for all $v\in (0,1)$,
$$\left[ n{{\nabla }_{v}}m,N{{\nabla }_{v}}M \right]\cap \left[ n,N \right],\left[ n{{\nabla }_{v}}m,N{{\nabla }_{v}}M \right]\cap \left[ m,M \right]=\varnothing,$$ 
then 
\begin{equation}\label{HH_main_ineq}
f\left( \frac{A+B}{2} \right)\le \int\limits_{0}^{1}{f\left( \left( 1-v \right)A+vB \right)dv}\le \frac{f\left( A \right)+f\left( B \right)}{2}	.
\end{equation}
\end{corollary}

\begin{proof}
As we have shown, if the assumptions of Theorem  \ref{6} are satisfied then
\begin{equation}\label{needed_HH_1}
f\left( \left( 1-v \right)A+vB \right)\le \left( 1-v \right)f\left( A \right)+vf\left( B \right),\quad\forall v\in (0,1).
\end{equation}
It follows from \eqref{needed_HH_1} that, for $v\in (0,1),$
\begin{equation}\label{needed_HH_2}
\begin{aligned}
f\left( \frac{A+B}{2} \right)&=f\left( \frac{\left( 1-v \right)A+vB+\left( 1-v \right)B+vA}{2} \right) \\ 
& \le \frac{1}{2}\left[ f\left( \left( 1-v \right)A+vB \right)+f\left( \left( 1-v \right)B+vA \right) \right] \\ 
& \le \frac{f\left( A \right)+f\left( B \right)}{2}. 
\end{aligned}
\end{equation}

Now, integrating over $v\in \left[ 0,1 \right]$ the inequalities in \eqref{needed_HH_2} and taking into account that
\[\int\limits_{0}^{1}{f\left( \left( 1-v \right)A+vB \right)dv}=\int\limits_{0}^{1}{f\left( \left( 1-v \right)B+vA \right)dv}\]
we infer 
\begin{equation}\label{needed_HH_3}
f\left( \frac{A+B}{2} \right)\le \int\limits_{0}^{1}{f\left( \left( 1-v \right)A+vB \right)dv}\le \frac{f\left( A \right)+f\left( B \right)}{2}	.
\end{equation}
Notice that \eqref{needed_HH_3} nicely extend the main result of \cite{drag_HH_1}.
\end{proof}

 Aujla showed that if $f:[0,\infty)\to [0,\infty)$ is an operator monotone decreasing function $f$, then \cite[Theorem 2.6]{2} 
 \begin{equation}\label{10}
2f\left( A+B \right)\le f\left( A \right)+f\left( B \right),
 \end{equation}
where $A$, $B$ are two positive matrices. Indeed, \eqref{10} follows by adding the following observations
\[f\left( A+B \right)\le f\left( A \right)\quad\text{ and }\quad f\left( A+B \right)\le f\left( B \right).\]
On account of Proposition \ref{oper_intro_prop}, we can improve \eqref{10} as follows.
\begin{proposition}
	Let $A,B\in \mathcal{B}\left( \mathcal{H} \right)$ be two positive operators. If $f:[0,\infty)\to [0,\infty)$ is an operator monotone decreasing function, then
\begin{equation}\label{9}
2f\left( A+B \right)\le 2f\left( A\nabla B \right)\le f\left( A \right)+f\left( B \right).
\end{equation}	
\end{proposition}
\begin{proof}
It follows from the assumption on $f$ that
\[A\nabla B\le A+B\quad\text{ }\Rightarrow\quad \text{ }f\left( A+B \right)\le f\left( A\nabla B \right).\]
It is known that such operator monotone decreasing function is also operator convex (see \cite[Theorem 2.4]{Uchiyama} and \cite{ando1}), i.e.,
\[f\left( A+B \right)\le f\left( A\nabla B \right)\le f\left( A \right)\nabla f\left( B \right)\]
 and the proof is complete. 
\end{proof}
\begin{remark}
Notice that \eqref{9} is also stronger than \cite[Proposition 3.14]{kian}, due to the fact that $f\left( A\nabla B \right)\le f\left( A \right)\sharp f\left( B \right)$, when $f:(0,\infty)\to (0,\infty)$ is operator monotone decreasing.
\end{remark}
In the following, we provide a reverse inequality for the subadditivity property of operator monotone decreasing functions. To reach this end, we need the following lemma which can be proved using the well-known Mond--Pe\v cari\'c method. A comprehensive survey on this topic can be found in  \cite[Chapter 2]{pe}.
\begin{lemma}\label{11}
Let ${{A}_{1}},\ldots ,{{A}_{n}}\in \mathcal{B}\left( \mathcal{H} \right)$ be  positive operators with spectra contained in $\left[ m,M \right]$ and ${{w}_{1}},\ldots ,{{w}_{n}}$ be positive scalars with $\sum\nolimits_{i=1}^{n}{{{w}_{i}}}=1$. If $f$ is a positive operator convex function on $\left[ m,M \right]$, then
\[\sum\limits_{i=1}^{n}{{{w}_{i}}f\left( {{A}_{i}} \right)}\le K\left( m,M,f \right)f\left( \sum\limits_{i=1}^{n}{{{w}_{i}}{{A}_{i}}} \right)\]
where 
\[K\left( m,M,f \right)=\max \left\{ \frac{1}{f\left( t \right)}\left( \frac{M-t}{M-m}f\left( m \right)+\frac{t-m}{M-m}f\left( M \right) \right):\text{ }t\in \left[ m,M \right] \right\}.\]
\end{lemma}
The next result is stated in terms of operator monotone decreasing function. Notice that this condition maybe replaced by operator convexity with the additional condition that $f(\infty)<\infty$, as we have seen in Proposition \ref{oper_intro_prop}.
\begin{theorem}\label{12}
	Let $A,B\in \mathcal{B}\left( \mathcal{H} \right)$ be two positive operators such that $n\le A,B\le N$. If $f:(0,\infty)\to (0,\infty)$ is an operator monotone decreasing function and  $0<m\leq 2n\leq 2N\leq M$, then
	\begin{equation}\label{13}
f\left( A \right)+f\left( B \right)\le 4K\left( m,M,f \right)f\left( A+B \right)	
	\end{equation}
where $K\left( m,M,f \right)$ is defined as in Lemma \ref{11}.
\end{theorem}
\begin{proof}
By Proposition \ref{oper_intro_prop}, $f$ is operator convex and  Lemma \ref{11} applies. By taking $n=2$, ${{w}_{1}}={{w}_{2}}={1}/{2}\;$, ${{A}_{1}}=2A$, and ${{A}_{2}}=2B$ in  Lemma \ref{11}, we get
\begin{equation*}
\frac{f\left( 2A \right)+f\left( 2B \right)}{2}\le K\left( m,M,f \right)f\left( A+B \right).
\end{equation*}
Since $f:(0,\infty)\to (0,\infty)$ is  operator monotone decreasing, we have $f\left( \alpha t \right)\ge \frac{1}{\alpha }f\left( t \right)$ for each $\alpha \ge 1$ by \cite[Lemma 2.2]{gms}. This implies the desired result. 
\end{proof}

A better estimate than that in Theorem \ref{12} may be obtained  for operator concave functions  as follows. To this end, an argument similar to that in Lemma \ref{11}  implies 
\[f\left( \sum\limits_{i=1}^{n}{{{w}_{i}}{{A}_{i}}} \right)\le \frac{1}{k\left( m,M,f \right)}\sum\limits_{i=1}^{n}{{{w}_{i}}f\left( {{A}_{i}} \right)}\]
where 
\begin{equation}\label{19}
k\left( m,M,f \right)=\min \left\{ \frac{1}{f\left( t \right)}\left( \frac{M-t}{M-m}f\left( m \right)+\frac{t-m}{M-m}f\left( M \right) \right):\text{ }t\in \left[ m,M \right] \right\},
\end{equation}
for the positive operator concave function $f$. Now we are ready to present a subadditive property for operator concave functions.\\
First, we notice that any non-negative concave function $f$ on $(0,\infty)$ satisfies the property $f(\alpha t)\leq \alpha f(t)$ for $\alpha\geq 1, t>0$. This can be easily seen by considering the derivative of the function $g(t)=f(\alpha t)-\alpha f(t).$ Of course, if $f$ is not differentiable, it is still can be approximated by smooth concave functions and the differential argument holds.
\begin{theorem}\label{sub_concave}
	Let $A,B\in \mathcal{B}\left( \mathcal{H} \right)$ be two positive operators such that $n\le A,B\le N$. 	
	If $f:(0,\infty)\to (0,\infty)$ is an operator  concave function and $0<m\leq 2n\leq 2N\leq M$, then
	\begin{equation}\label{130}
k\left( m,M,f \right)f\left( A+B \right)\le f\left( A \right)+f\left( B \right)
	\end{equation}
where $k\left( m,M,f \right)$ is as in \eqref{19}.
\end{theorem}
\begin{proof}
Proceeding like Theorem \ref{12} and noting that $f(\alpha t)\leq \alpha f(t)$ when $f$ is operator concave and $\alpha\geq 1$ , we obtain
\begin{align*}
f(A+B)&=f\left(\frac{2A+2B}{2}\right)\\
&\leq \frac{1}{k(m,M,f)}\frac{f(2A)+f(2B)}{2}\\
&\leq \frac{1}{k(m,M,f)}\left(f(A)+f(B)\right),
\end{align*}
completing the proof.
\end{proof}

\begin{remark}
The inequalities \eqref{13} and \eqref{130} can be extended in the following way:
	\[f\left( \sum\limits_{i=1}^{\ell}{{{A}_{i}}} \right)\le \frac{1}{k\left( m,M,f \right)}\sum\limits_{i=1}^{\ell}{f\left( {{A}_{i}} \right)},\]
	whenever $f$ is  operator concave on $\left[ m,M \right]$ and $n\le {{A}_{i}}\le N$, where $m\leq \ell\;n\leq \ell\;N\leq M$. In addition, 
	\[\frac{1}{{{\ell}^{2}}K\left( m,M,f \right)}\sum\limits_{i=1}^{\ell}{f\left( {{A}_{i}} \right)}\le f\left( \sum\limits_{i=1}^{\ell}{{{A}_{i}}} \right)\]
	whenever $f$ is an operator monotone decreasing on $\left[ m,M \right]$.
\end{remark}
We conclude our discussion of subadditive-type inequalities by the following versions for convex and concave functions (not necessarily operator convex or operator concave).

For this, we remind the reader of the following simple observations. If $f$ is convex (concave) on an interval $J$ containing the spectrum of a self adjoint operator $A$, then
\begin{equation}\label{conv_conc_inner}
f\left(\left<Ax,x\right>\right)\leq (\geq )\left<f(A)x,x\right>,
\end{equation}
 for any unit vector $x\in \mathcal{H}.$ In \cite[Theorem 6]{mis}, reverses of these celebrated inequalities were found as follows
\begin{equation}\label{rev_conv}
K(m,M,f)f\left(\left<Ax,x\right>\right)\geq  \left<f(A)x,x\right>,
\end{equation}
where $f$ is convex on $[m,M]$ and $m\leq A\leq M.$ On the other hand, if $f$ is concave on $[m,M]$ and $m\leq A\leq M$, we have
\begin{equation}\label{rev_conc}
f\left(\left<Ax,x\right>\right)\leq k(m,M,f)\left<f(A)x,x\right>.
\end{equation}
Utilizing \eqref{rev_conc} and \eqref{rev_conv} implies the following superadditive and subadditive versions. 
\begin{theorem}
Let $m,M$ be be positive scalars and $m\leq A,B\leq M.$ If $f:[0,M]\to [0,\infty)$ is convex with $f(0)=0$, then
$$f(A)+f(B)\leq K(m,M,f)f(A+B),$$ where $K\left( m,M,f \right)$ is defined as in Lemma \ref{11}. On the other hand, if $f$ is concave, then
$$f(A)+f(B)\geq k(m,M,f)f(A+B),$$ where $k\left( m,M,f \right)$ is defined as in \eqref{19}.
\end{theorem}
\begin{proof}
We prove the first inequality for convex function. The second inequality follows similarly. Let $x\in\mathcal{H}$ be a unit vector. Then, for the convex $f$ with $f(0)=0,$
\[\begin{aligned}
 \left\langle \left( f\left( A \right)+f\left( B \right) \right)x,x \right\rangle &=\left\langle f\left( A \right)x,x \right\rangle +\left\langle f\left( B \right)x,x \right\rangle  \\ 
& \le K\left( m,M,f \right)\left( f\left( \left\langle Ax,x \right\rangle  \right)+f\left( \left\langle Bx,x \right\rangle  \right) \right)\quad   ({\text{by}}\;\eqref{rev_conv}) \\ 
& \le K\left( m,M,f \right)f\left( \left\langle \left( A+B \right)x,x \right\rangle  \right) \quad ({\text{by}}\;\eqref{conc_sub_intro}\;{\text{for\;convex\;functions}})\\ 
& \le K\left( m,M,f \right)\left\langle f\left( A+B \right)x,x \right\rangle \quad   ({\text{by}}\;\eqref{conv_conc_inner}).
\end{aligned}\]
Since this is true for an arbitrary unit vector $x$, the first inequality follows immediately. The second follows similarly.
\end{proof}

\section*{Acknowledgement} The authors would like  to express their gratitude to the anonymous referee for valuable comments and corrections.

\vskip 0.4 true cm

{\tiny (H.R. Moradi) Young Researchers and Elite Club, Mashhad Branch, Islamic Azad
	University, Mashhad, Iran. }

{\tiny \textit{E-mail address:} hrmoradi@mshdiau.ac.ir}

{\tiny \vskip 0.4 true cm }

{\tiny (Z. Heydarbeygi) Department of Mathematics, Payame Noor Universtiy (PNU), P.O. BOX, 19395--4697, Tehran, Iran.
	
	\textit{E-mail address:} zheydarbeygi@yahoo.com}
	
{\tiny \vskip 0.4 true cm }

{\tiny (M. Sababheh) Department of Basic Sciences, Princess Sumaya University for Technology, Amman 11941,
	Jordan. 
	
\textit{E-mail address:} sababheh@psut.edu.jo}
%-----------------------------------------------------------------------------
%-----------------------------------------------------------------------------
\end{document}